\documentclass[a4paper]{article}
\usepackage{hyperref,authblk,amsmath,amssymb,amsthm,color}
\usepackage{pgf,tikz}
\usepackage{mathrsfs}
\usetikzlibrary{arrows}
\usepackage{circuitikz}

\title{Lights Out on graphs}
\author[1]{Abraham Berman} 
\author[2]{Franziska Borer} 
\author[3]{Norbert Hungerb\"uhler}
\affil[1]{Department of Mathematics, Technion - Israel Institute of Technology, 
Haifa 32000, Israel}
\affil[2]{Institute of Mathematics, Goethe University, 60325 Frankfurt am Main, Germany}
\affil[3]{Department of Mathematics, ETH Z\"urich, 8092 Z\"urich, Switzerland}
\date{}

\newtheorem{theorem}{Theorem}
\newtheorem{proposition}[theorem]{Proposition}
\newtheorem{lemma}[theorem]{Lemma}

\theoremstyle{definition}
\newtheorem{definition}[theorem]{Definition}

\begin{document}
\maketitle
\begin{abstract}\noindent
We model the Lights Out game on general simple graphs in the framework of 
linear algebra over the field $\mathbb F_2$. Based upon a version
of the Fredholm alternative, we introduce a separating invariant
of the game, i.e., an initial state can be transformed into a final
state if and only if the invariant of both states agrees. We also investigate
certain states with particularly interesting properties. Apart from the
classical version of the game, we propose several variants, in particular
a version with more than only two states (light on, light off), where
the analysis resides on systems of linear equations over the ring 
$\mathbb Z_n$. Although it is easy to find a concrete solution 
of the Lights Out problem, we show that it is NP-hard to find
a minimal solution. We also propose electric circuit diagrams
to actually realize the Lights Out game. 
\end{abstract}
\section{Introduction}
The game {\em Lights Out\/} exists in several versions. The
classical edition, issued by Tiger Electronics  in 1995, consists of a $5\times 5$ 
grid of lights which are switches at
the same time. Pressing a light will toggle the state of this light and its 
adjacent neighbors between off and on. Initially a random set of lights are
on and the aim is to turn all lights off, if possible with a minimum
number of operations.

This game and some variants of it have been studied several times in the literature: 
The classical problem as described above was modeled in the
language of linear algebra in~\cite{anderson-feil}.
In~\cite{goshima-yamagishi}
the puzzle on a $2^k\times 2^k$ torus is investigated, and a criterion for the solvability 
of the $5^k\times 5^k$ torus is derived. 
In~\cite{kreh} the solvability of the Lights Out game with respect to
varying board size and with more than one  color (see Section~\ref{sec:colored} below)  
is studied by methods from algebraic number theory.
In the present paper, we consider the Lights Out game and
some variants on general graphs.

We can play the game with the rules described above on an arbitrary simple graph $G$ with vertex set
$V:=\{v_1,\ldots,v_n\}$. Each vertex is a light and at the same time a switch.
The status of the lights is modeled by a map
$V\to\mathbb F_2=\{0,1\}$, represented by a vector $(x_1,\ldots,x_n)^\top$, 
where $x_i=1$ means that the light in vertex $v_i$ is on,
and off if $x_i=0$. Let $A\in\mathbb F_2^{n\times n}$ be the
adjacency matrix of $G$. Then $N:=A+\mathbb I$, where $\mathbb I$ is
the $n\times n$ identity matrix, encodes the information about
which lights are toggled when pushing a certain button:
If $x\in\mathbb F_2^n$ represents a state of lights and we press
switch $v_j$, then the resulting state will be 
\begin{equation}\label{eq-switch}
x+Ne_j,
\end{equation} 
where $e_j$ is the $j$-th column of $\mathbb I$, and where  
the operations in~(\ref{eq-switch}) are carried out in the field $\mathbb F_2$.
Hence, pressing a sequence of switches, say $v_{j_1}, v_{j_2},\ldots,v_{j_k}$,
will lead an initial state $x\in\mathbb F_2^n$ to the resulting state
$$
x+Ne_{j_1}+Ne_{j_2}+\ldots+Ne_{j_k}=x+N(e_{j_1}+e_{j_2}+\ldots+e_{j_k}).
$$
In particular, 
\begin{itemize}
\item the order in which switches are pushed does not play
a role, and
\item  pressing a switch an even number of times is equivalent to
not touching it at all.
\end{itemize}
Hence, for an initial state $i\in \mathbb F_2^n$ of lights, 
a given final state $f\in\mathbb F_2^n$ can be reached in the Lights Out game 
if and only if there exists $a\in\mathbb F_2^n$
such that $f=i+Na$. Here, $a=(a_1,a_2,\ldots,a_n)^\top$ encodes the concrete solution: 
To get from $i$ to $f$ we push all switches $v_j$ for which $a_j=1$.
\section{A separating invariant for the Lights Out game on graphs}
As we have seen in the introduction, a final state  $f\in\mathbb F_2^n$ can be reached from
an initial state $i\in\mathbb F_2^n$ if and only if $i+f$ belongs to the column
space of $N$.
This observation can conveniently be expressed by introducing the following
equivalence relation:
\begin{definition}
Two elements $x,y\in\mathbb F_2^n$ are called equivalent with respect to the
matrix $N\in\mathbb F_2^{n\times n}$ (or just equivalent) if $x+y$ lies in the column space of
$N$. In this case we write $x\sim y$.
\end{definition}
Hence, an initial state $i\in\mathbb F_2^n$ can be transformed into a final
state $f\in\mathbb F_2^n$ if and only if $i\sim f$. In principal, it is now
easy to check, if $i$ and $f$ belong to the same equivalence class 
by applying Gauss elimination in $\mathbb F_2$ for the linear system
$Na=i+f$. However, if we want to see at a glance whether $i\sim f$, 
we need a handy separating invariant. In order to formulate such an invariant, 
we first define a bilinear form on $\mathbb F_2^n$:
\begin{definition}
$\langle\cdot,\cdot\rangle: \mathbb F_2^n\times \mathbb F_2^n\to \mathbb F_2,
(x,y)\mapsto\sum_{i=1}^n x_i y_i$.
\end{definition}
Two vectors $x,y\in \mathbb F_2^n$ with $\langle x,y\rangle=0$ are called orthogonal.
If $M$ is an arbitrary subset of $\mathbb F_2^n$, $M^\bot := \{ x\in\mathbb F_2^n\mid \langle x,m\rangle=0$ 
for all $m\in M\}$ is a vector space.
Note however, that the bilinear form $\langle\cdot,\cdot\rangle$ is 
not positive definite: Every $x\in\mathbb F_2^n$
with an even number of ones satisfies $\langle x,x\rangle=0$ (such vectors are called {\em isotropic}).
Nonetheless, the {\em Fredholm alternative\/} holds in the following sense:
\begin{proposition}\label{fredholm}
Let $N\in\mathbb F_2^{m\times n}$ be an $m\times n$ matrix.  
Then the linear system $Nx=b$ has a solution $x\in\mathbb F_2^n$ if and only if
$b\in\mathbb F_2^m$ is orthogonal to all solutions of the adjoint system $N^\top y=0$.
I.e., there holds
$$\operatorname{im} N=(\operatorname{ker}N^\top)^\bot.$$
\end{proposition}
\begin{proof}
Let $C:=\operatorname{im}N$. Then $\operatorname{dim}C=\operatorname{rank} N=:r$.
Observe, that $\operatorname{ker}N^\top=C^\bot$ and that $\operatorname{dim}C^\bot=m-\operatorname{dim}C=m-r$.
Clearly, we have 
\begin{equation}\label{botbot}
C\subset C^\bot{}^\bot.
\end{equation} 
On the other hand, as above, we get for the dimension of $C^\bot{}^\bot$:
\begin{equation}\label{dimbotbot}
\operatorname{dim}C^\bot{}^\bot=m-\operatorname{dim}C^\bot=m-(m-r)=r.
\end{equation}
Thus, by~(\ref{botbot}) and~(\ref{dimbotbot}) the two vector spaces $\operatorname{im} N$ and $(\operatorname{ker}N^\top)^\bot$ agree.
\end{proof}
In the theory of linear codes, Proposition~\ref{fredholm} is formulated as $C^\bot{}^\bot = C$: a code $C$
coincides with its double-dual code.

Now we apply Proposition~\ref{fredholm} in our case to the symmetric matrix 
$N=N^\top\in\mathbb F_2^{n\times n}$ and get:
\begin{theorem}\label{thm-inv}
Let $A\in\mathbb F_2^{n\times n}$ be the adjacency matrix
of a simple graph $G$ and $N:=A+\mathbb I$ with rank $r$.
Let $v_1,\ldots, v_{n-r}$ be a basis of $\operatorname{ker}N$
and $J\in\mathbb F_2^{(n-r)\times n}$ be the matrix with rows $v_1^\top,\ldots, v_{n-r}^\top$.
Then $x\mapsto Jx$ is a separating invariant for the Lights Out game on $G$:
An initial state $i$ can be transformed into a final state $f$ if and only if
$Ji=Jf$.  In particular, all lights can be turned off if and only if $Ji=0$.
\end{theorem}
\begin{proof}
As we have seen above, $i$ can be transformed into $f$ if and only if
$i+f\in\operatorname{im} N$. According to Proposition~\ref{fredholm} this
is equivalent to the fact, that $i+f$ is orthogonal to $\operatorname{ker}N$,
which in turn means that $J(i+f)=0$. And over $\mathbb F_2$ this 
is equivalent to $Ji=Jf$.
\end{proof}
\section{Special states}\label{sec-special}
\subsection{Inverting a state}\label{inverting}
Let $x\in\mathbb F_2^n$ describe a state in the Light Out game.
Then we say that $\bar x:=x+(1,1,\ldots,1)^\top$ is the {\em inverse state}:
Lights which are on in $x$ are off in $\bar x$ and vice versa.
Now, we claim that in the Lights Out game every state can be inverted:
\begin{theorem}\label{inverse}
Let $G$ be an arbitrary simple graph and $x$ an initial state. Then,
$x$ can be transformed into its inverse state $\bar x$.
\end{theorem}
In order to prove Theorem~\ref{inverse} we start with the following Lemma:
\begin{lemma}\label{lemma}
Let $N\in\mathbb F_2^{n\times n}$ be symmetric with all diagonal elements $N_{ii}=1$.
Then, each $x\in\operatorname{ker}N$ is isotropic.
\end{lemma}
\begin{proof}
Let $x\in\operatorname{ker}N$. We have to show, that the number
of components $x_i$ of $x$ which are equal to 1 is even. Let $\Omega=\{j\mid x_j=1\}$
and $N^{(i)}$ denote the $i$-th column of $N$. Then 
\begin{equation}\label{j}
0=\langle N^{(i)},x\rangle = \sum_{j\in\Omega}N_{ij}.
\end{equation}
Taking the sum over $i\in\Omega$ in (\ref{j}) we get, by the symmetry of $N$,
\begin{equation}\label{ij}
0=\sum_{i,j\in\Omega}N_{ij}=\sum_{i\in\Omega}N_{ii}+2\sum_{i< j\in\Omega}N_{ij}.
\end{equation}
The first term on the right hand side of~(\ref{ij}) is $\operatorname{card}(\Omega)\operatorname{mod}2$ because all 
diagonal elements of $N$ equal 1,
the second term is zero in $\mathbb F_2$. Hence $x$ is isotropic.
\end{proof}
\begin{proof}[Proof of Theorem~\ref{inverse}]
Let $A\in\mathbb F_2^{n\times n}$ be the adjacency matrix
of  $G$ and $N:=A+\mathbb I$.
We have to show that $x+\bar x=(1,1,\ldots,1)^\top\in \operatorname{im}N=(\operatorname{ker}N)^\bot$.
In fact, according to Lemma~\ref{lemma}, all elements $y\in\operatorname{ker}N$
have an even number of ones. Hence, $\langle (1,1,\ldots,1)^\top,y\rangle=0$.
\end{proof}
Let $\Theta=\{v_{j_1},\ldots,v_{j_k}\}$ be a set of buttons which, when pressed,
invert a state. Such a set will be called {\em inverting set}. 
Hence, if $x\in\mathbb F_2^n$ is a vector with ones at the
coordinates $j_1,\ldots,j_k$ and zeros at all other coordinates, there holds
$Nx=(1,\ldots,1)^\top$. Such a vector will be called {\em inverting}.

{\bf Remark.} Theorem~\ref{inverse} is optimal in the sense
that for the complete graph $K_n$, we have $\operatorname{dim}(\operatorname{im}N)=1$. 
Indeed, on $K_n$ inverting a state is the only possible operation.
\subsection{Self-reproducing, self-avoiding and neutral vectors}
Is it possible, starting from all lights off, to press a certain set of buttons such that
exactly those lights are turned on? Such a {\em self-reproducing set\/} $\Theta=\{v_{j_1},
\ldots,v_{j_k}\}$ corresponds to a {\em self-reproducing vector} $x\in\mathbb F_2^n$ having ones
exactly at the coordinates $j_\ell$, $\ell=1,\ldots,k$, and zeros in all other coordinates,
such that $Nx=(A+\mathbb I)x=x$. Hence we have the following:
\begin{proposition}\label{self-reproducing}
A vector $x\in\mathbb F_2^n$ is  self-reproducing on the graph $G$
with adjacency matrix $A$ if and only if $Ax=0$.
\end{proposition}
Is it possible, starting from all lights off, to press a certain set of buttons such that
exactly all other lights are turned on? Such a {\em self-avoiding set\/} $\Theta=\{v_{j_1},
\ldots,v_{j_k}\}$ corresponds to a {\em self-avoiding vector} $x\in\mathbb F_2^n$ having ones
exactly at the coordinates $j_\ell$, $\ell=1,\ldots,k$, and zeros in all other coordinates,
such that $Nx=(A+\mathbb I)x=\bar x=x+(1,\ldots,1)^\top$. Hence we have:
\begin{proposition}
A vector $x\in\mathbb F_2^n$ is  self-avoiding on the graph $G$
with adjacency matrix $A$ if and only if $Ax=(1,\ldots,1)^\top$.
\end{proposition}
Next we are looking for a set of lights $\Theta=\{v_{j_1},
\ldots,v_{j_k}\}$ with the following property: If the lights in
$\Theta$ are on and all others are off, and if then 
all buttons in $\Theta$ are pushed, then again exactly the lights in $\Theta$
are supposed to be on. Such a {\em neutral set\/} $\Theta=\{v_{j_1},
\ldots,v_{j_k}\}$ corresponds to a {\em neutral vector} $x\in\mathbb F_2^n$ having ones
exactly at the coordinates $j_\ell$, $\ell=1,\ldots,k$, and zeros in all other coordinates,
such that $x=x+Nx=x+(A+\mathbb I)x$, i.e. $Nx=0$ or equivalently $Ax=x$. 
\begin{proposition}
A vector $x\in\mathbb F_2^n$ is  neutral on the graph $G$
with adjacency matrix $A$ if and only if $Ax=x$.
\end{proposition}
In other words,  pressing the buttons which correspond to a neutral vector $x$
has no effect since $x=x+0=x+Nx$.
From these considerations we conclude immediately: 
\begin{proposition}
\begin{itemize}
\item The sum of an even number of self-avoiding vectors is self-reproducing.
\item The sum of  an arbitrary number of self-reproducing vectors is self-reproducing.
\item The sum of  an arbitrary number of neutral vectors is neutral.
\item The sum of a self-avoiding vector and an arbitrary number of self-reproducing vectors is self-avoiding.
\item The sum of a neutral vector and a sum of an even number of inverting vectors is neutral.
\end{itemize}
\end{proposition}
\section{Variants of the game}
\subsection{The Second Neighbors Lights Out game}
Here, we change the rules of the game a little bit to create a new challenge.
Let $G$ be a simple graph with the property, that either every vertex has even
degree or zero degree, or every vertex has odd degree.
In the first case every row of the adjacency matrix $A$ is isotropic. On this graph,
we play the Lights Out game as follows: If we push a button $v_i$, the corresponding
light is inverted and also the second neighbors. However, if a vertex $v_k$ is connected with $v_i$
by an even number of paths of length 2, then $v_k$ keeps its status.
We call this game {\em the Second Neighbors Lights Out game}.
We find:
\begin{theorem}\label{variant}
Let $A\in\mathbb F_2^{n\times n}$ be the adjacency matrix
of a simple graph $G$ such that 
\begin{itemize}
\item[(a)] all vertices have even degree or zero degree, and $r$ is the rank of $\bar N:=A^2+\mathbb I$, or
\item[(b)] all vertices have odd degree, and $r$ is the rank of $\bar N:=A^2$.
\end{itemize}
Let $v_1,\ldots, v_{n-r}$ be a basis of $\operatorname{ker}\bar N$
and $J\in\mathbb F_2^{(n-r)\times n}$ be the matrix with rows $v_1^\top,\ldots, v_{n-r}^\top$.
Then $x\mapsto Jx$ is a separating invariant for the Second Neighbors Lights Out game on $G$:
An initial state $i$ can be transformed into a final state $f$ if and only if
$Ji=Jf$.
\end{theorem}
\begin{proof}
Observe, that the Second Neighbors Lights Out game is equivalent to the original Lights Out game,
but on a new simple graph $\bar G$ with adjacency matrix $\bar A:=A^2$ in case (a).
It is at this point where we need the fact, that all rows of $A$ are isotropic: This implies
that all elements on the diagonal of $A^2$ are 0, which means that indeed $\bar G$
has no loops. Let $\bar N=\bar A+\mathbb I  =A^2+\mathbb I$. Then, the claim follows from
Theorem~\ref{thm-inv} applied to $\bar G$.
Similarly in case (b) $\bar A=A^2+\mathbb I$ is the adjacency matrix of a simple
graph $\bar G$. Hence, if we put $\bar N=\bar A+\mathbb I=A^2$, then the claim follows from
Theorem~\ref{thm-inv} applied to $\bar G$ as above.
\end{proof}

Moreover, we have
\begin{theorem}
\begin{enumerate}
\item[(a)] Every state in the Second Neighbors Lights Out game on $G$ can be inverted.
\item[(b)] If all vertices of $G$ have even degree or zero degree, then the following
is true: If an initial state $i$ can be transformed into a final state $f$ in the Second
Neighbors Lights Out game on $G$, then this is also possible for the original
Lights Out game on $G$.
\end{enumerate}
\end{theorem}
\begin{proof}
Part (a) follows from Theorem~\ref{inverse} applied to the graph $\bar G$
in the proof of Theorem~\ref{variant}. 

To prove part (b) we argue as follows: 
Let again $A$ be the adjacency matrix of $G$, and
$N:=A+\mathbb I$. If, as in the proof of Theorem~\ref{variant},  $\bar A=A^2$ denotes the adjacency matrix 
of $\bar G$, we have $\bar N:=\bar A+\mathbb I=A^2+\mathbb I=N^2$.
Recall that the Second Neighbors Lights Out game on $G$ is equivalent to
the original Lights Out game on $\bar G$. Thus,
an initial state $i$ can be transformed into 
a final state $f$ in the Second Neighbors Lights Out game on $G$ if and only if
$i+f$ belongs to the column space of $\bar N$. But $\operatorname{im}N^2
\subset\operatorname{im}N$, and hence $i$ can be transformed into $f$ 
in the original Lights Out game on $G$ if this is possible on $\bar G$.
\end{proof}
\subsection{The Neighborhood Lights Out game}
Here, we mix the original and the Second Neighbors Lights  Out game:
Let $G$ be a simple graph 
such that either every vertex has odd degree, or every vertex has even degree or zero degree. On this graph,
we play the game as follows: If we push a button $v_i$, the corresponding
light is inverted, its direct neighbors and also the second neighbors. 
As before, if a vertex $v_k$ is connected with $v_i$
by an even number of paths of lengths 1 or 2, then $v_k$ keeps its status.
We call this game {\em the Neighborhood Lights Out game}.
For this game we have:
\begin{theorem}\label{neigborhood-game}
Let $A\in\mathbb F_2^{n\times n}$ be the adjacency matrix
of a simple graph $G$ such that 
\begin{itemize} 
\item all vertices have odd degree, and  $r$ is the rank of $\bar N:=A+A^2$, or
\item all vertices have even degree or zero degree, and  $r$ is the rank of $\bar N:=A+A^2+\mathbb I$.
\end{itemize}
Let $v_1,\ldots, v_{n-r}$ be a basis of $\operatorname{ker}\bar N$
and $J\in\mathbb F_2^{(n-r)\times n}$ be the matrix with rows $v_1^\top,\ldots, v_{n-r}^\top$.
Then $x\mapsto Jx$ is a separating invariant for the Neighborhood Lights Out game on $G$:
An initial state $i$ can be transformed into a final state $f$ if and only if
$Ji=Jf$.
\end{theorem}
\begin{proof}
In case (a) the Neighborhood Lights Out game is equivalent to the original Lights Out game,
but on a new simple graph $\bar G$ with adjacency matrix $\bar A:=A+A^2+\mathbb I$. 
Since all vertices of $G$ have odd degree
the elements on the diagonal of $\bar A$ are 0, and hence  $\bar G$
has no loops. Let $\bar N=\bar A+\mathbb I  =A+A^2$. Then, the claim follows from
Theorem~\ref{thm-inv} applied to $\bar G$.
Similarly in case (b), $\bar A=A+A^2$ is the adjacency matrix of a simple graph $\bar G$.
Hence, with $\bar N=\bar A+\mathbb I=A+A^2+\mathbb I$, the claim follows from
Theorem~\ref{thm-inv} applied to $\bar G$.
\end{proof}
Moreover, we have
\begin{theorem}
\begin{enumerate}
\item[(a)] Every state in the Neighborhood Lights Out game on $G$ can be inverted.
\item[(b)] If every vertex of $G$ has odd degree, then the following ist true: If an initial state $i$ can be transformed into a final state $f$ in the Neighborhood 
Lights Out game on $G$, then this is also possible for the original Lights Out game on $G$.
\end{enumerate}
\end{theorem}
\begin{proof}
Part (a) follows from Theorem~\ref{inverse} applied to the graph $\bar G$
in the proof of Theorem~\ref{neigborhood-game}. 

For part (b) let again $A$ be the adjacency matrix of $G$ and $N:=A+\mathbb I$. 
Then $\bar A=A+A^2+\mathbb I$ is the adjacency matrix of $\bar G$,
and $\bar N:=\bar A+\mathbb I=A+A^2=N+N^2$.
The Neighborhood Lights Out game on $G$ is equivalent to
the original Lights Out game on $\bar G$. Thus,
an initial state $i$ can be transformed into 
a final state $f$ in the Neighborhood Lights Out game on $G$ if and only if
$i+f$ belongs to the column space of $\bar N$. But $\operatorname{im}(N+N^2)
\subset\operatorname{im}N$, and hence $i$ can be transformed into $f$ 
in the original Lights Out game on $G$ if this is possible on $\bar G$.
\end{proof}
\subsection{The Non-reflexive Lights Out game}\label{non-reflexive}
Suppose, your Lights Out game device does not function properly
anymore: You notice that some of the buttons $v_i$ (maybe all) invert only
their direct neighbors but no longer the light in $v_i$. This means
that on the diagonal of the original matrix $N=A+\mathbb I$, some
places turned from 1 to 0. Hence the game is described by
a new matrix $\bar N=A+\bar{\mathbb I}$, where $\bar{\mathbb I}\in\mathbb F_2^{n\times n}$
is a given diagonal matrix.

It is clear that Theorem~\ref{thm-inv} with $N$ replaced by $\bar N$ remains true for this version of the
game. However, Theorem~\ref{inverse} turns into the following:
\begin{theorem}
Let $G$ be a simple graph with adjacency matrix $A$, and 
$\bar N=A+\bar{\mathbb I}$, where $\bar{\mathbb I}\in\mathbb F_2^{n\times n}$ is diagonal.
Then, $\bar{\mathbb I}(1,1,\ldots,1)^\top$ belongs to the column space of $\bar N$. 
I.e., if  $i,f\in\mathbb F_2^n$ are such that $i+f=\bar{\mathbb I}(1,1,\ldots,1)^\top$,
then the initial state $i$ can be transformed into the final state $f$.
In particular, if all lights are turned off initially, the state $\bar{\mathbb I}(1,1,\ldots,1)^\top$
can be reached.
\end{theorem}
\begin{proof}
Let  
$d=(\bar N_{11}, \bar N_{22},\ldots,\bar N_{nn})^\top=\bar{\mathbb I}(1,\dots,1)^\top$ be the vector consisting of the diagonal entries of $\bar N$
and $x\in\ker \bar N$ be an element of the kernel of $\bar N$. Then, we have
\[ 0 = x^{\mathrm T}\bar N x = \sum_{i,j=1}^n\bar N_{ij}x_ix_j = 2\sum_{i<j}\bar N_{ij}x_ix_j+\sum_{i=1}^n\bar N_{ii}x_i^2 = \left<d,x\right>. \]
Thus, $d$ is orthogonal to all elements $x\in\ker \bar N$, which is, according to  Proposition~\ref{fredholm}
and in view of the symmetry of the matrix $\bar N$, equivalent to $d\in\operatorname{im}(\bar N)$.
\end{proof}
\subsection{The Colored Lights Out game}\label{sec:colored}
We consider again a simple graph $G$ with $n$ vertices and with  adjacency matrix $A$.
Instead of toggling between on and off, we may consider a cycle
\begin{center}
off $\to$ color 1 $\to$ color 2 $\to$ \ldots $\to$ color $k-1$ $\to$ off.
\end{center}
We encode the colors by their numbers, and off by $0$. Hence a state
of colored lights on $G$  corresponds to a vector $(x_1,\ldots,x_n)^\top\in\mathbb Z_k^n$, where
$x_i$ indicates the color of the light in vertex $v_i$. 
Pushing a button in vertex $v_i$ will increment
the color in $v_i$ and its neighbors by one, modulo $k$. Suppose we push
$a_i$-times the button in $v_i$, $i=1,\ldots,n$, then an initial state $i\in\mathbb Z_k^n$ of
colored lights on the graph will end up in the final state $f=i+Na$, 
where $a=(a_1,\dots,a_n)^\top$ and the operations are carried out  in $\mathbb Z_k$,
and where $N=A+\mathbb I$.
This means, we have to solve the linear system
\begin{equation}\label{smith0}
Na\equiv c\mod k
\end{equation}
where $c=f-i$. We first treat the simpler case where $k$ is squarefree.
\begin{proposition} 
If $k$ is a product of different prime numbers $p_1,p_2,\ldots,p_\ell$, then a solution of~(\ref{smith0})
can be found as follows:
\begin{enumerate}
\item Solve~(\ref{smith0}) modulo $p_i$ for $i=1,2,\ldots,\ell$ by Gauss elimination, i.e.,
find $b_i\in\mathbb Z^n$ such that $Nb_i\equiv c\mod p_i$ for $i=1,2,\ldots,\ell$.
\item
Let $k_i:=k/p_i$ and use the extended Euclidean algorithm to find $r_i,s_i\in\mathbb Z$ such that
$r_ip_i+s_i k_i=1$ for each $i$.
Then a solution of~(\ref{smith0}) is given by 
$$a=\sum_{i=1}^\ell b_is_ik_i.$$
\end{enumerate}
In particular, a solution of~(\ref{smith0}) exists if and only if $\operatorname{det}N\not\equiv 0 \mod p_i$
for all $i\in\{1,\ldots,\ell\}$.
\end{proposition}
\begin{proof}
Observe that $s_ik_i\equiv 1\mod p_i$ and $s_ik_i\equiv 0\mod p_j$ for all $j\neq i$.  Therefore we have
modulo $p_j$
$$
Na=\sum_{i=1}^\ell Nb_is_ik_i \equiv \sum_{i=1}^\ell (c+q_ip_i) s_ik_i\equiv c
$$
for all $j\in\{1,\ldots,\ell\}$, and the claim follows by the Chinese Remainder theorem.
\end{proof}
For the general case of an arbitrary $k\in \mathbb N$, 
we first observe, that~(\ref{smith0}) can  equivalently be written as
\begin{equation}\label{smith1}
(N\mid k\mathbb I)\bar a=c
\end{equation}
where $M:=(N\mid k\mathbb I)$ is the $n\times 2n$-matrix with $N$ in the left block and $k\mathbb I$
in the right block, and where $\bar a$ is the vector with first $n$ components $a$ and
$n$ additional (unknown) integer components below.
Then~(\ref{smith1}), and hence~(\ref{smith0}), can be solved as follows:
\begin{proposition}
Let $UMV=B$ be the Smith decomposition of $M$, i.e.,  $U\in\mathbb Z^{n\times n}$ and
$V\in\mathbb Z^{2n\times 2n}$ are unimodular matrices, and $B\in\mathbb Z^{n\times 2m}$
is the Smith normal form of $M$. Let $y$ be an integer solution of $By=Uc$.
Then, $\bar a=Vy$ solves~(\ref{smith1}), and in particular, the first $n$ components
of $\bar a$ solve~(\ref{smith0}). Moreover,~(\ref{smith1}) has a solution if
and only if $B_{ii}$ divides the $i$-th component of $Uc$ for all $i$.
\end{proposition}
\begin{proof}
The proof follows immediately from the properties of the Smith decomposition.
\end{proof}

\subsection{The asymmetric Lights Out game}
Here we play the game on a directed graph $G$: If a certain button is pressed
the corresponding light toggles and also the lights in all of its out-neighbors.
The adjacency matrix $A\in\mathbb F_2^{n\times n}$ of $G$ is in this case in general not symmetric:
$$A_{jk}=\begin{cases}
1 &\text{if $j$ is the foot and $k$ is the head of a directed edge,}\\
0 &\text{otherwise.}
\end{cases}
$$
We put again $N=A+\mathbb I$. 
If $i\in\mathbb F_2^n$ is an initial state and $f\in\mathbb F_2^n$ a 
final state, then $i$ can be transformed into $f$ if and only if $f=i+N^\top a$ for 
some $a\in\mathbb F_2^n$, i.e., if $i+f$ belongs to the column space of $N^\top$.
 But then, according to Proposition~\ref{fredholm}, $\operatorname{im}N^\top=(\operatorname{ker}N)^\bot$,
 and therefore Theorem~\ref{thm-inv} remains  true for the Lights Out game on the directed graph $G$.

\section{The problem of finding minimal solutions}
The questions we encountered so far could easily be answered by
standard tools of linear algebra: E.g.~to determine whether an initial
state $i$ can be transformed into a final state $f$ for the Lights Out
game is answered by the separating invariant from Theorem~\ref{thm-inv}:
This requires only to compute a matrix multiplication $J(i+f)$ and to
check if the result is the zero vector. Similarly, to actually compute
a solution for this problem, we just apply Gauss elimination 
to solve the linear system $Na=i+f$ for $a$ over $\mathbb F_2$.
However, it turns out that finding minimal solutions is computationally
much more difficult. For example we have:
\begin{proposition}
The problem of finding a self-reproducing vector $x\neq 0$ in Proposition~\ref{self-reproducing} 
of minimal Hamming weight is NP-hard.
\end{proposition}
\begin{proof}
Recall that a self-reproducing vector $x$ is characterized by the equation $Ax=0$,
where $A$ is the adjacency matrix of the graph $G$. Such a vector
$x$ corresponds to a vertex set $V'=\{v_{j_1},\ldots,v_{j_\ell}\}\subset V$
with $v_{j_k}\in V'\iff x_{j_k}=1$ which has the property that every vertex $v\in V$
has an even number of vertices of $V'$ (or none of them) among its neighbors. This translates into
the following decision problem:
\begin{center}
\begin{tabular}{lp{.7\linewidth}}
\multicolumn{2}{l}{\bf Even Vertex Set}\\
{\bf Instance:} & A graph $G=(V,E)$ and an integer $w>0$.\\
{\bf Question:} & Is there a nonempty subset $V'\subset V$ of at most $w$
vertices, such that every vertex $v\in V$ has an even number of 
vertices of $V'$ among its neighbors?
\end{tabular}
\end{center}
It has been shown by Vardy~\cite{vardy} that the Even Vertex Set problem
is NP-complete. To find a vector $x$ with $Ax=0$ and minimal Hamming
weight is therefore NP-hard.
\end{proof}
As a second example, we take a closer look at
the problem of finding a minimal solution to the 
Non-reflexive Lights Out game with $N=A$ (see Section~\ref{non-reflexive}), i.e.,
a solution with a minimum buttons pressed. Suppose our graph has $n$ vertices.
As we have seen above, to transform an initial state $i\in\mathbb F_2^n$
into a final state $f\in\mathbb F_2^n$ we need to solve the linear system
$Aa=i+f$ over $\mathbb F_2$ for $a\in \mathbb F_2^n$. If the matrix rank
of $A$ over $\mathbb F_2$ is $n$, the solution $a$ 
is unique. 
However, in general, $r:=\operatorname{rank}A<n$
and the set of solutions (if non-empty) forms an affine space 
$\mathcal A$ of dimension $n-r>0$ in
$\mathbb F_2^n$. To find a solution with a minimum number of buttons to
push means to find a vector in $\mathcal A$ with minimum Hamming weight.
We will show, that this problem is NP-hard. The reasoning is based upon
the result of Berlekamp, McEliece and van Tilborg who
showed in~\cite{berlekamp} that the nearest codeword problem is NP-complete:
\begin{center}
\begin{tabular}{lp{.7\linewidth}}
\multicolumn{2}{l}{\bf Nearest Codeword Problem}\\
{\bf Instance:} & A binary $m\times n$ matrix $A$, a binary vector $y\in\mathbb F_2^m$, and an integer $w>0$.\\
{\bf Question:} & Is there a vector $x\in\mathbb F_2^n$ of Hamming weight $\le w$ such that $Ax=y$?
\end{tabular}
\end{center}
We first show that the  special case of a square matrix $A$ is not easier than the general problem:
\begin{center}
\begin{tabular}{lp{.7\linewidth}}
\multicolumn{2}{l}{\bf Balanced Nearest Codeword Problem}\\
{\bf Instance:} & A binary $n\times n$ matrix $A$, a binary vector $y\in\mathbb F_2^n$, and an integer $w>0$.\\
{\bf Question:} & Is there a vector $x\in\mathbb F_2^n$ of Hamming weight $\le w$ such that $Ax=y$?
\end{tabular}
\end{center}
\begin{lemma}
The Balanced Nearest Codeword Problem is NP-complete.
\end{lemma}
\begin{proof}
It is easy to see that the Balanced Nearest Codeword Problem is in NP. 
We will now show that the General Nearest Codeword Problem can be
reduced to the Balanced Nearest Codeword Problem: Let $A\in\mathbb F_2^{m\times n}$ and
$y\in\mathbb F_2^m$. Consider first the case $n<m$ and let
$$
\bar A:=(A\mid 0)\in\mathbb F_2^{m\times m}, \quad \bar y:=y.
$$
If  $x$ is a  solution of $Ax=y$ with Hamming weight $w$, then
$$
\bar x=\left(\begin{array}{c}x\\\hline 0\end{array}\right)\in\mathbb F_2^{m}
$$
is a solution of $\bar A\bar x=\bar y$ with Hamming weight $w$. Vice versa,
if $\bar x\in\mathbb F_2^m$ is a solution of $\bar A\bar x=\bar y$, then 
$$
x=\begin{pmatrix}\bar x_1\\\vdots\\\bar x_n\end{pmatrix}\in\mathbb F_2^{n}
$$
is a solution of $Ax=y$ with Hamming weight $\le w$.

If $n>m$ we define similarly
$$
\bar A:=\left(\begin{array}{c} A\\\hline 0\end{array}\right)\in\mathbb F_2^{n\times n}, \quad
\bar y:=\left(\begin{array}{c} y\\\hline 0\end{array}\right)\in\mathbb F_2^{n}. 
$$
If  $x$ is a  solution of $Ax=y$ with Hamming weight $w$, then
$
\bar x=x
$
is a solution of $\bar A\bar x=\bar y$ with Hamming weight $w$ and vice versa.
\end{proof}
Next, we show that the symmetric case is not easier than the general problem:
\begin{center}
\begin{tabular}{lp{.7\linewidth}}
\multicolumn{2}{l}{\bf Symmetric Nearest Codeword Problem}\\
{\bf Instance:} & A symmetric binary $n\times n$ matrix $A$, a binary vector $y\in\mathbb F_2^n$, and an integer $w>0$.\\
{\bf Question:} & Is there a vector $x\in\mathbb F_2^n$ of Hamming weight $\le w$ such that $Ax=y$?
\end{tabular}
\end{center}
\begin{lemma}\label{lem-np}
The Symmetric Nearest Codeword Problem is NP-complete.
\end{lemma}
\begin{proof}
Obviously the Symmetric Nearest Codeword Problem is in NP, and we
have to show that the Balanced Nearest Codeword Problem can be
reduced to the Symmetric Nearest Codeword Problem. To this end, let $A\in\mathbb F_2^{n\times n}$ and
$y\in\mathbb F_2^n$, and consider
$$
\bar A=\left(\begin{array}{c|c}0&A^\top\\\hline A&0\end{array}\right)=\bar A^\top\in\mathbb F_2^{2n\times 2n},\quad 
\bar y=\left(\begin{array}{c}0\\\hline y\end{array}\right)\in\mathbb F_2^{2n}.
$$
If $x$ is a solution of $Ax=y$ with Hamming weight $w$, then
$$
\bar x=\left(\begin{array}{c}x\\\hline 0\end{array}\right)\in\mathbb F_2^{2n}
$$
is a solution of $\bar A\bar x=\bar y$ with the same Hamming weight $w$.
Vice versa, if $\bar x\in\mathbb F_2^{2n}$ is a solution of $\bar A\bar x=\bar y$ with Hamming weight $w$, then
$$
x=\begin{pmatrix}\bar x_1\\\vdots\\\bar x_n\end{pmatrix}\in\mathbb F_2^{n}
$$
is a solution of $Ax=y$ with Hamming weight $\le w$. 
\end{proof}
The corresponding computational problem for the Non-reflexive Lights Out game is the following:
\begin{center}
\begin{tabular}{lp{.7\linewidth}}
\multicolumn{2}{l}{\bf Non-reflexive Lights Out game Problem}\\
{\bf Instance:} & A simple graph $G$ with $n$ vertices, an initial state $i\in\mathbb F_2^n$,
a final state $f\in\mathbb F_2^n$, and an integer $w>0$.\\
{\bf Question:} & Is there a solution $a$ in $\mathbb F_2^n$ of Hamming weight $\le w$ such that $i+Ax=f$?
\end{tabular}
\end{center}
From Lemma~\ref{lem-np} it follows now immediately:
\begin{theorem}
The  Non-reflexive Lights Out game Problem is NP-complete. Hence, to find
a solution of the Non-reflexive Lights Out game Problem with a minimal
number of pressed buttons is NP-hard.
\end{theorem}
If the rank of $A$ is $r<n$, this indicates that one
has to go over all of the $2^{n-r}$ points in the affine space defined by $Ax=i+f$ to determine a 
minimal solution of the Non-reflexive Lights Out game.

\section{Circuit diagrams of the Lights Out game}\label{circuit}
It is far from obvious how a classical electric circuit with voltage source, lamps and switches
can be built which realizes the Lights Out game. We propose such a circuit diagram for
the cycle graph $\mathcal C_n$ with $n\ge 3$ vertices. To this end, consider the following building
block $\mathcal B_k$:
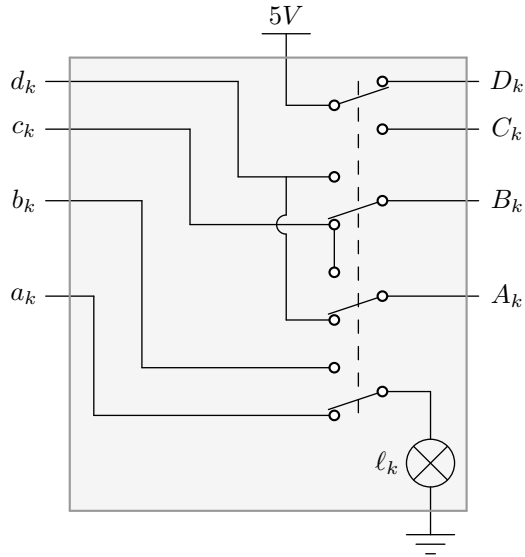
\begin{center}
\definecolor{aqaqaq}{rgb}{0.6,0.6,0.6}
\begin{tikzpicture}[line width=.5,line cap=round,line join=round,>=triangle 45,x=9,y=9]
\fill[color=aqaqaq,fill=aqaqaq,fill opacity=0.1] (-11,-4) -- (5.5,-4) -- (5.5,15) -- (-11,15) -- cycle;
\draw  (2.,1.)-- (-0.246,0.246);
\draw (2.,5.)-- (-0.246,4.246);
\draw  (2.,1.)-- (4.,1.);
\draw  (2.,9.)-- (-0.246,8.246);
\draw  (0.,13.)-- (2.246,13.754);
\draw  (4.,-2.) circle (1);
\draw  (4.,1.)-- (4.,-1.);
\draw  (3.2928932188134525,-2.7071067811865475)-- (4.707106781186548,-1.2928932188134525);
\draw  (4.707106781186548,-2.7071067811865475)-- (3.2928932188134525,-1.2928932188134525);
\draw  (3.,-5.)-- (5.,-5.);
\draw  (0.,13.)-- (-2.,13.);
\draw  (-2.,13.)-- (-2.,16.);
\draw  (-3.,16.)-- (-1.,16.);
\draw  (0.,6.)-- (0.,8.);
\draw  (0.,10.)-- (-4.,10.);
\draw  (-4.,10.)-- (-4.,14.);
\draw  (0.,4.)-- (-2.,4.);
\draw  (-10.,0.)-- (-10.,5.);
\draw  (-8.,2.)-- (-8.,9.);
\draw  (0.,0.)-- (-10.,0.);
\draw  (0.,2.)-- (-8.,2.);
\draw  (-8.,9.)-- (-12.,9.);
\draw  (0.,8.)-- (-6.,8.);
\draw  (-6.,8.)-- (-6.,12.);
\draw  (-6.,12.)-- (-12.,12.);
\draw  (-4.,14.)-- (-12.,14.);
\draw [dash pattern=on 5 off 5] (1.,14.)-- (1.,0.);
\draw [shift={(-2.,8.)}]  plot[domain=1.5707963267948966:4.71238898038469,variable=\t]({1.*0.4*cos(\t r)+0.*0.4*sin(\t r)},{0.*0.4*cos(\t r)+1.*0.4*sin(\t r)});
\draw  (-2.,4.)-- (-2.,7.6);
\draw  (-2.,8.4)-- (-2.,10.);
\draw  (2.,5.)-- (6.,5.);
\draw  (2.,9.)-- (6.,9.);
\draw  (2.,12.)-- (6.,12.);
\draw  (2.,14.)-- (6.,14.);
\draw  (4.,-3.)-- (4.,-5.);
\draw  (3.8,-5.8)-- (4.2,-5.8);
\draw  (3.4,-5.4)-- (4.6,-5.4);
\draw  (-10.,5.)-- (-12.,5.);
\draw (-12.9,12) node {$c_k$};
\draw (-12.9,14) node {$d_k$};
\draw (7.2,5) node {$A_k$};
\draw (7.2,9) node {$B_k$};
\draw (7.2,12) node {$C_k$};
\draw (7.2,14) node{$D_k$};
\draw[color=black] (2.2,-2.) node {$\ell_k$};
\draw[color=black] (-2.,16.8) node {$5V$};
\draw[color=black] (-12.9,9) node {$b_k$};
\draw[color=black] (-12.9,5) node {$a_k$};
\draw [line width=.8,fill=white,fill opacity=1.0] (0.,0.) circle (0.2);
\draw [line width=.8,fill=white,fill opacity=1.0] (0.,2.) circle (0.2);
\draw [line width=.8,fill=white,fill opacity=1.0] (0.,4.) circle (0.2);
\draw [line width=.8,fill=white,fill opacity=1.0] (0.,6.) circle (0.2);
\draw [line width=.8,fill=white,fill opacity=1.0] (0.,8.) circle (0.2);
\draw [line width=.8,fill=white,fill opacity=1.0] (0.,10.) circle (0.2);
\draw [line width=.8,fill=white,fill opacity=1.0] (2.,14.) circle (0.2);
\draw [line width=.8,fill=white,fill opacity=1.0] (2.,12.) circle (0.2);

\draw [line width=.8,fill=white,fill opacity=1.0] (2.,1.) circle (0.2);
\draw [line width=.8,fill=white,fill opacity=1.0] (0.,13.) circle (0.2);
\draw [line width=.8,fill=white,fill opacity=1.0] (2.,5.) circle (0.2);
\draw [line width=.8,fill=white,fill opacity=1.0] (2.,9.) circle (0.2);

\draw [line width=.8,color=aqaqaq] (-11.,-4.)-- (5.5,-4);
\draw [line width=.8,color=aqaqaq] (5.5,-4)-- (5.5,15.);
\draw [line width=.8,color=aqaqaq] (5.5,15.)-- (-11.,15.);
\draw [line width=.8,color=aqaqaq] (-11.,15.)-- (-11.,-4.);
\end{tikzpicture}
\end{center}
Then, combine blocks $\mathcal B_1,\ldots,\mathcal B_n$ such that $A_k$ is
connected to $a_{k+1}$, $B_k$ to $b_{k+1}$, $C_k$ to $c_{k+1}$, and $D_k$ to  $d_{k+1}$,
for $k=1,\ldots,n-1$, and close the cycle by connecting $A_n$ and $a_1$, $B_n$ and $b_1$, 
$C_n$ and $c_1$, and $D_n$ and $d_1$. Then, the quadrupole changeover switch in
block $\mathcal B_k$ will invert the lamps $\ell_{k},\ell_{k+1}$ and $\ell_{k+2}$ cyclically.
An initial state can be arranged by changing, if necessary, the contacts of the switch next to each lamp.

Using a circuit diagram with boolean gates gives more flexibility. In this
case, the building blocks are T flip-flops:
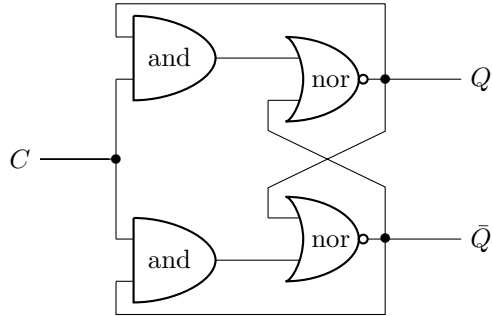
\begin{center}
\begin{circuitikz}
\draw%
(0,0) node[anchor=center,and port] (myand1) {\hspace*{-3.5mm}\raisebox{-1mm}{and}}
(0,2.68) node[and port] (myand2) {\hspace*{-3.5mm}\raisebox{-1mm}{and}}
(2,.28) node[nor port] (mynor1) {\hspace*{-2mm}\raisebox{-1mm}{nor}}
(2,2.4) node[nor port] (mynor2) {\hspace*{-2mm}\raisebox{-1mm}{nor}}
(myand1.out) -- (mynor1.in 2)
(myand2.out) -- (mynor2.in 1)
(mynor1.out) --++ (1, 0) node [anchor = west] {$\bar Q$}
(mynor2.out) --++ (1, 0) node [anchor = west] {$Q$}
(mynor1.out) --++(0,-1) --++ (-3.54,0) --++ (0,.44) -- (myand1.in 2)
(mynor2.out) --++(0,1)  --++ (-3.54,0) --++ (0,-.44) -- (myand2.in 1)
(myand1.in 1)--(myand2.in 2) node [anchor = center, inner sep=0pt, pos=.5] (E) {\small\textbullet}
(E) --++(-1,0) node [anchor = east] {$C$} --++(1.05,0)

(mynor1.in 1) --++ (0,.4) --++ (1.54,.75) -- (mynor2.out)
(mynor2.in 2) --++ (0,-.4) --++ (1.54,-.75) -- (mynor1.out)
;
\node  at (mynor1.out) {\small\textbullet};
\node  at (mynor2.out) {\small\textbullet};
\end{circuitikz}
\end{center}
Whenever a clock pulse at $C$ arrives, the output state $Q,\bar Q=\neg Q$ toggles.
To build the Lights Out game, we are going to use only the $Q$ output which 
will represent the lamp: $Q=1$ corresponds to light on, $Q=0$
corresponds to light off. To implement the Lights Out game on an arbitrary
graph, we represent every vertex $v_i$ by the following circuit element:
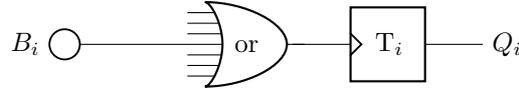
\begin{center}
\begin{circuitikz}\draw
(0,0) node[or port] (myor) {\hspace*{-1.5mm}\raisebox{-1mm}{or}}
(myor) to [twoport] (2.5,0) 
;
\node  at (1.27,0) {T$_i$};
\node [anchor = west] at (2.5,0) {$Q_i$};
\node [anchor = east] at (-3.2,0) {$B_i$};

\draw [line width=.7] (.753,-.15)-- (.903,0)-- (.753,.15);
\draw [line width=.4] (-3.,0)-- (-.93,0);
\draw [line width=.4] (-1.386,-.14)-- (-.94,-.14);
\draw [line width=.4] (-1.386,.14)-- (-.94,.14);

\draw [line width=.4] (-1.386,-.42)-- (-1.03,-.42);
\draw [line width=.4] (-1.386,.42)-- (-1.03,.42);

\draw [line width=.7,fill=white,fill opacity=1.0] (-3.,0.) circle (0.2);
\end{circuitikz}
\end{center}
The output of the or-gate is connected to the clock input
of the T flip-flop T$_i$. 
The button $B_i$ of vertex $v_i$  is then connected to 
\begin{itemize}
\item an input pin of the corresponding or-gate (see figure above), and
\item an input pin of the or-gate of every neighboring vertex of $v_i$.  
\end{itemize}
In this way, pushing the button $B_i$ toggles the outputs $Q_i$ and
the outputs of all neighboring vertices.
\section{Lights Out on graphs with symmetries}
Let  $G=(V,E)$ be a graph with vertex set $V=\{v_1,\ldots,v_n\}$ and edge set $E=\{e_1,\ldots,e_k\}$, 
where each edge $e_\ell$ is the set $\{v_i,v_j\}$ of the two vertices which are incident with it.
An automorphism of $G$ is a bijective map $$\phi:V\to V\text{ such that }
\{\phi(v_i),\phi(v_j)\}\in E\iff \{v_i,v_j\}\in E.$$ 
The bijection $\phi:(v_1,\ldots,v_n)\mapsto
(v_{\sigma (1)},\ldots,v_{\sigma(n)})$
corresponds to a permutation $\sigma$ of the index set $\{i,\ldots,n\}$ which in turn
can be expressed by a permutation matrix $P\in\mathbb F_2^{n\times n}:
(1,\ldots,n)^\top\mapsto P(1,\ldots,n)^\top=(\sigma(1),\ldots,\sigma(n))^\top$ with the property  
that $P^\top AP=A$.

Prominent examples of graphs with a nontrivial automorphism group
are Cayley graphs: 
Let $\mathcal G$ be a finite group and $S\subset \mathcal G$ be a
symmetric subset, i.e., with each $s\in S$, also the inverse $s^{-1}\in S$.
The Cayley graph $\Gamma(\mathcal G,S)$ has as vertices the elements $g\in\mathcal G$
and the edges are $\{\{g,gs\}:g\in \mathcal G,s\in S\}$. 
If we assume in addition, that $S$ does not contain
the unit element of $\mathcal G$, the resulting graph has no loops, and is
therefore suitable for the Lights Out game.
Observe that $\mathcal G$ is a subgroup of the automorphism group
of the Cayley graph $\Gamma(\mathcal G,S)$. In fact, each $h\in\mathcal G$
induces an automorphism of $\Gamma(\mathcal G,S)$ by $\mathcal G\to\mathcal G,
g\mapsto hg$. Indeed, the edge $\{g,gs\}$ is mapped to the edge $\{(hg),(hg)s\}$.

The symmetries of a graph allow to map a special state of the Lights Out game
from Section~\ref{sec-special} into a similar one by applying an automorphism of the graph: 
\begin{proposition}
Let $\phi$ be an automorphism of a graph. Then there holds:
\begin{enumerate}
\item[(a)] If $J$ is an inverting set, then $\phi(J)$ is an inverting set.
\item[(b)] If $J$ is a self-avoiding set, then $\phi(J)$ is a self-avoiding set.
\item[(c)] If $J$ is a self-reproducing set, then $\phi(J)$ is a self-reproducing set.
\item[(d)] If $J$ is a neutral set, then $\phi(J)$ is a neutral set.
\end{enumerate}
\end{proposition}
\begin{proof}
(a) Let $x\in\mathbb F_2^n$ be the vector with ones at the coordinates
which correspond to elements of $J$, and $y=P^\top x$
the vector with ones at the coordinates
which correspond to elements of $\phi(J)$.
$J$ is inverting if and only if $Nx=(1,\ldots,1)^\top$. 
By putting $x = PP^\top x$ we get, after multiplying the previous equation 
from the left by $P^\top$, the identity
\begin{eqnarray*}
(1,\ldots,1)^\top&=&P^\top (1,\ldots,1)^\top=
P^\top Nx=\\&=&P^\top (A+\mathbb I)PP^\top x = 
P^\top (A+\mathbb I)Py=(A+\mathbb I)y=Ny 
\end{eqnarray*}
and hence $\phi(J)$ is inverting.

(b), (c) and (d) are analogous.
\end{proof}

\section*{Acknowledgment}
We would like to thank Jonathan Hungerb\"uhler for his help with 
the circuit diagrams in Section~\ref{circuit}.
We would also like to thank Peter Elbau for helpful discussions.

\end{document}